\documentclass[12pt,twoside,a4paper]{amsart}
\usepackage[textwidth=15cm,textheight=22cm,centering]{geometry}
\usepackage[T1]{fontenc}
\usepackage{lmodern,microtype}
\usepackage{amsmath,amssymb,amsthm,mathtools,mathrsfs}
\usepackage{xurl}
\usepackage[hidelinks]{hyperref}
\hypersetup{pdftitle={Answers to three problems in Dinh-Sibony's list on holomorphic foliation theory},pdfauthor={Zhangchi Chen}}
\newtheorem{maintheorem}{Theorem}

\newtheorem{lemma}{Lemma}[section]
\theoremstyle{definition}
\newtheorem{openproblem}{Open problem}[section]
\newenvironment{dsproblem}[1]{\par\medskip\noindent
  \textbf{Problem #1 \textup{(Dinh--Sibony)}.}\ \itshape}{\par\medskip}
\newcommand{\ddc}{dd^c}
\newcommand{\F}{\mathcal F}
\newcommand{\D}{\mathbb D}
\newcommand{\abs}[1]{\lvert#1\rvert}
\newcommand{\norm}[1]{\lVert#1\rVert}
\newcommand{\pair}[2]{\langle#1,#2\rangle}
\DeclareMathOperator{\dist}{dist}
\title[Three problems in Dinh-Sibony's list]{Answers to three problems in Dinh-Sibony's list on holomorphic foliation theory}
\author{Zhangchi Chen}
\date{}
\subjclass[2020]{Primary 37F75; Secondary 32Q30, 32U40, 37A30.}
\keywords{Holomorphic foliation, harmonic current, Poincar\'e mass,
Nevanlinna characteristic, quantitative equidistribution.}
\address{Zhangchi Chen, School of Mathematical Sciences, Key Laboratory of MEA
(Ministry of Education), and Shanghai Key Laboratory of PMMP, East
China Normal University, Shanghai 200241, China.}
\email{zcchen@math.ecnu.edu.cn}
\begin{document}
\begin{abstract}
We address Problems~2.2(i), 2.2(ii), and~2.4 in Dinh--Sibony's list.
For a saturated holomorphic curve foliation on a compact K\"ahler
manifold with finitely many singularities, every positive directed
$dd^c$-closed current has finite Poincar\'e mass. Thus linearizability
is not needed for the geometric ergodic theorem in this setting.

We construct a projective surface with dense hyperbolic leaves and a
diffuse extremal current whose Nevanlinna averages converge at a uniform
reciprocal logarithmic rate. This is an existence result for a chosen
foliation, not a rate theorem for an arbitrary prescribed foliation.

Under the hypotheses of the first result, every hyperbolic leaf has
a logarithmic lower bound for its Nevanlinna characteristic, even if
no diffuse directed current exists. Positive-dimensional singular sets
are not treated.
\end{abstract}
\maketitle

\section{Introduction}

The Poincar\'e metric relates the intrinsic geometry of a hyperbolic leaf
to the distribution of its image in a compact complex manifold.
Dinh--Sibony's problem list \cite{DSProblems} asks how this relation
behaves near nonlinearizable singularities, whether geometric
equidistribution can be made quantitative, and whether the resulting
growth estimate holds for every leaf. We treat these three questions.
We first recall them with their original numbers, then state our
results and specify the extent of each answer.

\subsection{The three problems}

In the setting of \cite[Section~2]{DSProblems}, let $\F$ be a possibly
singular holomorphic foliation by Riemann surfaces on a compact K\"ahler
manifold $(M,\omega)$. A directed current is a positive current of
bidimension $(1,1)$ tangent to the regular foliation. We use
$dd^c=(i/\pi)\partial\bar\partial$ and
$\|T\|_\omega=\langle T,\omega\rangle$. All currents in our results
are defined on $M$, and their $dd^c$-closedness is understood on $M$.

For a hyperbolic regular leaf $L_a$, choose a universal covering
$\phi_a\colon\D\to L_a$, with $\phi_a(0)=a$. Set
\begin{equation}\label{eq:intro-nevanlinna}
 \mathcal T_a(r)=\int_0^r\frac{dt}{t}
                         \int_{\D(t)}\phi_a^*\omega,
 \qquad
 \tau_r^a=\frac1{\mathcal T_a(r)}
          \int_0^r\frac{dt}{t}(\phi_a)_*[\D(t)].
\end{equation}
The geometric ergodic theorem quoted as Theorem~2.1 in the problem
list states the following. If the singularities are linearizable and
$T$ is an extremal positive directed $dd^c$-closed current of mass one
whose trace gives full mass to hyperbolic regular leaves, then
$\tau_r^a\to T$ for $(T\wedge\omega)$-almost every $a$.
The underlying theorem of Dinh--Nguy\^en--Sibony uses finite Poincar\'e
mass; their linearizable-singularity estimate verifies this hypothesis
\cite[Proposition~4.2 and Theorems~7.1, 7.7]{DNS2012}.

The following are the three questions, reformulated in this notation.
No finiteness assumption on the singular set is inserted into their
statements.

\begin{dsproblem}{2.2(i)}
Does the geometric convergence theorem remain valid when the
singularities are not assumed to be linearizable?
\end{dsproblem}

\begin{dsproblem}{2.2(ii)}
Is there an extremal positive $dd^c$-closed current directed by the
foliation for which the speed of convergence of $\tau_r^a$ against
smooth $(1,1)$-forms can be estimated?
\end{dsproblem}

The problem list suggests currents of positive entropy as candidates
for the second question. This is a suggested approach, rather than an
additional hypothesis in its statement.

The quoted convergence theorem yields, for almost every hyperbolic
leaf, an estimate
\begin{equation}\label{eq:original-growth}
       \mathcal T_a(r)\ge c_a\log\frac1{1-r}
\end{equation}
for $r$ sufficiently close to $1$, where $c_a>0$ may depend on $a$.

\begin{dsproblem}{2.4}
Does \eqref{eq:original-growth} hold for every hyperbolic leaf, including
when there is no diffuse positive directed $dd^c$-closed current?
\end{dsproblem}

Here ``hyperbolic'' makes explicit the disk-covering setting of the
characteristic in the original question. A current is diffuse if it
gives no mass to proper analytic subsets.

\subsection{The three results}

We now state our hypotheses. A holomorphic curve foliation is taken to
be saturated: the quotient of $TM$ by its rank-one tangent sheaf is
torsion-free. In Theorems~\ref{main:mass} and~\ref{main:growth}, the
singular set $E$ is finite. No linearization or nondegeneracy condition
is imposed at its points. On a hyperbolic leaf we normalize the
Poincar\'e form by
\[
 \phi_a^*\omega_{\mathrm P}
       =dd^c[-\log(1-|\zeta|^2)],
\]
and set it equal to zero on nonhyperbolic leaves.

\begin{maintheorem}[Problem 2.2(i)]\label{main:mass}
Let $\F$ be a saturated holomorphic curve foliation with finite singular
set $E$ on a compact K\"ahler manifold $(M,\omega)$. There is a constant
$C=C(M,\omega,\F)$ such that every positive directed $dd^c$-closed
current $T$ satisfies
\[
 \int_{M\setminus E}T\wedge\omega_{\mathrm P}
                      \le C\|T\|_\omega.
\]
If $T\ne0$ is extremal and its trace gives full mass to hyperbolic
regular leaves, then
\[
 \tau_r^a\longrightarrow T/\|T\|_\omega
 \quad\text{for }(T\wedge\omega)\text{-almost every }a.
\]
\end{maintheorem}

\begin{maintheorem}[Problem 2.2(ii)]\label{main:rate}
There exist a smooth projective surface $(M,\Omega)$, a nonsingular
holomorphic curve foliation with every leaf dense and hyperbolic, and a
smooth diffuse positive closed current $T$ of mass one, extremal in the
cone of positive directed $dd^c$-closed currents, such that
\[
 \big|\langle\tau_r^a-T,\theta\rangle\big|
       \le\frac{C\|\theta\|_{C^{11}}}{-\log(1-r^2)}
\]
for every smooth $(1,1)$-form $\theta$, every $a\in M$, and $0<r<1$.
The constant depends only on the constructed geometry and a fixed
choice of the $C^{11}$ norm.
\end{maintheorem}

\begin{maintheorem}[Problem 2.4]\label{main:growth}
Let $\F$ satisfy the hypotheses on the foliation in
Theorem~\ref{main:mass}. There is a constant $c_0>0$, independent of
the hyperbolic leaf and its center, such that for every such center $a$
there is $C_a\ge0$ with
\[
 \mathcal T_a(r)\ge c_0[-\log(1-r^2)]-C_a,
 \qquad 0<r<1.
\]
In particular, \eqref{eq:original-growth} holds for every hyperbolic
leaf. No existence assumption on a diffuse directed current is needed.
\end{maintheorem}

\subsection{Extent of the answers}

\emph{Problem~2.2(i).}
Theorem~\ref{main:mass} removes linearizability throughout the class of
saturated foliations with finite isolated singular sets, including
degenerate singularities. On compact complex surfaces, saturation makes
the singular set finite, so no further restriction on its dimension is
needed there. In higher ambient dimension, the singular set may have
positive dimension. The original Section~2 does not explicitly exclude
that possibility, and our theorem does not treat it. Thus the answer
to this part is stated in the isolated singularity class.

\emph{Problem~2.2(ii).}
Theorem~\ref{main:rate} constructs both $\F$ and $T$. It gives an
affirmative example for the existence interpretation of the question,
with the additional properties that $T$ is diffuse, every leaf is
dense, and the estimate is uniform in the center. It does not prove
the stronger assertion that, for every foliation fixed beforehand in
the setting of the convergence theorem, one can find a current with
an explicit rate. We retain this distinction because the original
question refers to a current directed by the given foliation.
Our surface is projective but is not a projective space; the constructed
current is closed and the transverse holonomy consists of translations.

\emph{Problem~2.4.}
Theorem~\ref{main:growth} covers every hyperbolic leaf, with no
almost-everywhere exception, and applies even in the absence of diffuse
directed currents. It answers the question in the isolated linearizable
setting inherited from the preceding convergence theorem and extends
it to arbitrary isolated singularities. In the convention used in
\cite[Section~4]{DNS2012}, a linearizable singularity is modeled by a
diagonal linear vector field with all eigenvalues nonzero; it is
therefore isolated. If Problem~2.4 is instead read in the unrestricted
singular setting of the opening paragraph of Section~2, the additional
case of a positive-dimensional singular set is not settled here.

These are the precise senses in which the paper gives three answers.
We do not assert that all three questions are settled under every
possible interpretation of their ambient hypotheses and quantifiers.

\subsection{Earlier work and proof strategy}

Finite Poincar\'e mass at linearizable isolated singularities was proved
by Dinh--Nguy\^en--Sibony \cite[Proposition~4.2]{DNS2012}.
The local metric estimates of Lins Neto and Canille Martins
\cite[Theorem~3(a)]{CMLN2004}, together with the mass criterion in
\cite[Proposition~4.5]{BacherHeat}, also cover nondegenerate
singularities without a linearization assumption. Thus nonlinearizability
alone is not the distinction in Theorem~\ref{main:mass}. Its finite-mass
assertion includes arbitrary degenerate isolated singularities and does
not assume Brody hyperbolicity. Its convergence conclusion is an
application of the published geometric ergodic theorem
\cite[Theorems~7.1 and~7.7]{DNS2012}.

The suspension in Theorem~\ref{main:rate} also has a geometric
precedent. Matsumoto \cite[Example~4.1]{Matsumoto2016} proves
qualitative equidistribution of Poincar\'e disks for a class of
homogeneous foliations that includes dense torus-translation
suspensions. Moreover, the leafwise differential in our construction
is a real tube structure: partial Fourier series and Diophantine
conditions for its global solvability are studied by Ara\'ujo,
Ferra, Jahnke, and Ragognette \cite{AFJR2024}. We give an explicit
quantitative application of this approach to Nevanlinna currents on a
projective product. The scalar Poisson inverse is estimated in a fixed
$C^{11}$ norm, and the resulting reciprocal-logarithmic bound is uniform
in the center. The cited first-order closed-range theorem does not
itself state this scalar inverse estimate or this rate.

For Theorem~\ref{main:growth}, the nonsingular case is already known:
Dinh--Sibony \cite[Section~2]{DSProblems} explicitly attribute the
every-leaf logarithmic lower bound on compact nonsingular laminations
to Forn\ae ss--Sibony \cite{FS2008}. They distinguish this from the
invariant-line setting of \cite{DSInvariant}, where the characteristic
is known to be unbounded but the stated logarithmic estimate is not
provided. Theorem~\ref{main:growth} extends the every-leaf lower bound
to the finite isolated singularity setting, including degenerate
singularities. It is a consequence of the intrinsic metric estimate
and Jensen's formula, independent of the existence of a directed
current and of the almost-everywhere convergence theorem.

The intrinsic metric argument has an explicit antecedent in
Brunella \cite[Proposition~6.2]{Brunella}: long complex-time flow disks
near a zero of the defining vector field force its Poincar\'e norm to
tend to zero. We use this observation for ordinary regular leaves,
without invoking the completed-leaf framework used in his global
uniformization theory. A related use of singular metric potentials
appears in Nguy\^en's cohomological formula for Poincar\'e mass
\cite[Theorem~1.2 and Section~5]{NguyenMass}, for diffuse extremal
currents on Brody-hyperbolic projective surfaces with hyperbolic
singularities. Those hypotheses do not cover arbitrary degenerate
isolated singularities.

Here the logarithm of the ratio of the intrinsic norm to a smooth
reference norm is Borel and smooth along each hyperbolic leaf. Its
bounded truncations are constant near the singular set. Harmonic
disintegration and leafwise Stokes then give Theorem~\ref{main:mass},
whereas Jensen's formula gives Theorem~\ref{main:growth} directly.
For Theorem~\ref{main:rate}, Jensen's formula is applied instead to a
bounded Poisson potential; the characteristic is exactly
$-\log(1-r^2)$. Leafwise potentials thus connect the three arguments,
while their conclusions concern finite mass, a quantitative example,
and growth on every leaf, respectively.

We prove Theorems~\ref{main:mass}, \ref{main:rate}, and~\ref{main:growth}
in Sections~\ref{sec:finite-poincare-mass}, \ref{sec:rate}, and
\ref{sec:every-leaf}, respectively. Section~\ref{sec:further-cases}
formulates open problems suggested by the cases not settled here.

\section{Finite Poincar\'e mass and isolated singularities}
\label{sec:finite-poincare-mass}

Throughout this section, $(M,\omega,\F,E)$ satisfies the hypotheses
of Theorem~\ref{main:mass}. A current of bidimension $(1,1)$ has
bidegree $(m-1,m-1)$ in ambient dimension $m$. The equation $dd^cT=0$
is understood across $E$, not merely on its complement.

On a hyperbolic regular leaf, write $\|\cdot\|_{\mathrm P}$ for the
curvature $-1$ Poincar\'e length.  We normalize its associated leafwise
form by requiring that on the universal covering disk
\begin{equation}\label{eq:fm-poincare-normalization}
 \omega_{\mathrm P}
 =\frac{i}{\pi}\frac{d\zeta\wedge d\bar\zeta}
                         {(1-|\zeta|^2)^2}
 =dd^c[-\log(1-|\zeta|^2)].
\end{equation}
Thus $\omega_{\mathrm P}$ is $1/(2\pi)$ times the usual curvature $-1$
area.  Put both the Poincar\'e pseudonorm and this form equal to zero on
nonhyperbolic leaves.  The leafwise form is Borel and locally bounded on
the regular locus \cite[Proposition~3.1]{DNS2012}.

\begin{lemma}\label{lem:intrinsic-poincare-potential}
The tangent sheaf of $\mathcal F$ is a holomorphic line bundle $L$,
with a holomorphic morphism $\iota:L\to TM$ whose zero set is $E$.
Choose a smooth Hermitian length norm $\|\cdot\|_0$ on $L$ and set
\[
 \varphi=\log\frac{\|\cdot\|_{\mathrm P}}{\|\cdot\|_0}
\]
on hyperbolic regular leaves, using $\iota$ to define the numerator.
Put $\varphi=-\infty$ on nonhyperbolic leaves and on $E$.
Then $\varphi$ is Borel, bounded above, and tends uniformly to
$-\infty$ at every point of $E$.  If $e$ is a local nonvanishing
holomorphic frame of $L$, the forms
\[
                 \Theta_0=dd^c\log\|e\|_0
\]
define a global smooth real $(1,1)$-form.  On each hyperbolic leaf,
\begin{equation}\label{eq:fm-curvature}
        dd^c_{\mathcal F}\varphi
          =\omega_{\mathrm P}-\Theta_0|_{T\mathcal F}.
\end{equation}
\end{lemma}

\begin{proof}
Let $\mathscr L\subset TM$ be the saturated rank-one tangent sheaf.
The inclusion induces an injection $\mathscr L^{**}\to TM$; its kernel
vanishes because it vanishes generically and $\mathscr L^{**}$ is
torsion-free.  The quotient $\mathscr L^{**}/\mathscr L$ is torsion and
embeds into $TM/\mathscr L$, which is torsion-free by saturation.
Thus $\mathscr L=\mathscr L^{**}$.  A rank-one reflexive sheaf on a
smooth complex manifold is locally free, since the regular local rings
are factorial.  This proves the first assertion.

Fix $s\in E$, a coordinate ball centered at $s$, and a local frame $e$.
The vector field $v=\iota(e)$ vanishes at $s$, so after shrinking the
ball,
\[
                      |v(z)|\le C_0|z-s|.
\]
For $p$ sufficiently close to $s$, the holomorphic flow $\Phi_p$ of $v$
through $p$ exists on a complex time disk of radius
\[
             R(p)=\frac{1}{2C_0}\log\frac{\delta}{|p-s|}.
\]
Indeed, along each radial time segment, Gronwall's inequality gives
$|\Phi_p(t)-s|\le |p-s|e^{C_0|t|}$.  For $|t|<R(p)$ this remains
strictly inside a fixed smaller coordinate ball.  More explicitly, solve $\gamma'(\sigma)=t v(\gamma(\sigma))$, $0\le\sigma\le1$,
with $\gamma(0)=p$, treating $t$ as a complex parameter. The same estimate keeps
the solution in a compact subball for $|t|<R(p)$; holomorphic dependence
on $t$ defines a single holomorphic flow map on that disk.  It cannot meet $s$: uniqueness at such a time would make it the
constant solution.  The image consequently lies in the uncompleted
regular leaf through $p$.

If this leaf is hyperbolic, lift $\Phi_p$ to its universal covering disk
and apply Schwarz--Pick at zero.  Since $\Phi_p'(0)=v(p)$,
\begin{equation}\label{eq:fm-flow-bound}
                     \|e(p)\|_{\mathrm P}\le\frac{2}{R(p)}.
\end{equation}
No injectivity of the flow disk is needed.  On a nonhyperbolic leaf the
left side is zero by convention.  The reference norm of $e$ is bounded
below, so \eqref{eq:fm-flow-bound} proves uniform decay of $\varphi$ at
$s$.  The local boundedness and measurability of the Poincar\'e form,
together with a finite flow-box cover of the complement of fixed singular
balls, prove the global upper bound and Borel measurability.

Frame changes add $\log|g|$, with $g$ nowhere zero and holomorphic, to
$\log\|e\|_0$; hence $\Theta_0$ is well defined.  On the covering disk,
\[
 dd^c\log\|\partial_\zeta\|_{\mathrm P}
 =dd^c\log\frac{2}{1-|\zeta|^2}
 =\omega_{\mathrm P}.
\]
Subtracting the corresponding identity for the reference norm gives
\eqref{eq:fm-curvature}.
\end{proof}

We next record the integration statement needed because $\varphi$ need
not be continuous transversely.

\begin{lemma}\label{lem:bounded-leafwise-stokes}
Let $A\ge0$ and let $T$ be a positive directed $dd^c$-closed current.
Let $u$ be a
bounded Borel function on $M$, continuous along regular leaves, equal
to one constant on an ambient neighborhood of $E$.  Suppose that
$A\omega|_{T\mathcal F}+dd^c_{\mathcal F}u\ge0$ on every regular
plaque, in the distributional sense.  Then the local measures
\[
           T\wedge(A\omega+dd^c_{\mathcal F}u)
\]
extend and glue to a positive measure $\mu_u$ on $M$, with
\begin{equation}\label{eq:fm-exact-mass}
                     \mu_u(M)=A\langle T,\omega\rangle.
\end{equation}
\end{lemma}

\begin{proof}
In a regular flow box, harmonic disintegration gives
$T=\int h_t[P_t]\,d\nu(t)$, where the $h_t$ are nonnegative harmonic
functions \cite[Proposition~2.3]{DNS2012}.  On a smaller box choose a
smooth leafwise potential $q$ for $A\omega$.  The functions $u_t+q_t$
are bounded and subharmonic.  Convolution in the leaf coordinate defines
their Riesz measures and gives, for a smooth compactly supported test
function $\psi$,
\begin{align}
 \langle\mu_u,\psi\rangle
 &=\int\left(A\int_{P_t}\psi h_t\omega
              +\int_{P_t}u_t\,dd^c(\psi h_t)\right)d\nu(t)
                                                        \label{eq:fm-local-wedge}\\
 &=A\langle T,\psi\omega\rangle
       +\int u\,d\bigl(dd^c(\psi T)\bigr).
                                                        \label{eq:fm-intrinsic-wedge}
\end{align}
Here $dd^c(\psi T)$ is a signed Radon measure.  This follows either
directly from the plaque formula, or from normality of $T$ and $dT$ in
Proposition~2.3, the product rule, and $dd^cT=0$.  The interior
Harnack and gradient estimates bound the variation of
$dd^c(\psi h_t)$ by a fixed multiple of the $h_t$-mass on a slightly
larger plaque.  This bound is integrable in $t$.  It justifies dominated
convergence after leafwise convolution and shows that
\eqref{eq:fm-local-wedge} is a positive, locally finite measure.  No
transverse derivative of $u$ occurs.  Formula
\eqref{eq:fm-intrinsic-wedge} is independent of the disintegration, so
the local measures agree on overlaps.

On a neighborhood where $u$ is constant, define $\mu_u=A\,T\wedge\omega$.
This agrees with the regular definition on the punctured overlap.  Choose
a smooth partition of unity $\chi_0+\sum_{j=1}^N\chi_j=1$ with
$\chi_0=1$ near $E$, its support contained in the constant neighborhood,
and each $\chi_j$, $j\ge1$, supported in a regular flow box.  Set
\[
 \lambda=\sum_{j=1}^Ndd^c(\chi_jT)
        =-dd^c(\chi_0T).
\]
The left expression shows that $\lambda$ is a signed Radon measure;
the right expression and $dd^cT=0$ show that it is supported in the
regular region where $\chi_0$ varies.  Moreover,
$\lambda(M)=0$ by compact Stokes.  Since $u$ is constant on this support,
\[
 \mu_u(M)-A\langle T,\omega\rangle
       =\int u\,d\lambda=0.
\]
This proves \eqref{eq:fm-exact-mass} without pairing a Borel function
with a distribution at a singular point.
\end{proof}

\begin{proof}[Proof of Theorem~\ref{main:mass}]
Choose $A>0$ so that $A\omega-\Theta_0\ge0$ on $M$ and set
\[
                         u_n=\max\{\varphi,-n\}.
\]
Each $u_n$ is bounded and Borel, continuous along every regular leaf,
and equal to $-n$ on an ambient neighborhood of $E$.
On hyperbolic plaques, \eqref{eq:fm-curvature} and the maximum principle
for quasi-subharmonic functions give
$A\omega+dd^c_{\mathcal F}u_n\ge0$; on nonhyperbolic plaques $u_n=-n$.
Lemma~\ref{lem:bounded-leafwise-stokes} therefore gives positive measures
$\mu_n$ of mass $A\langle T,\omega\rangle$.

Let $H$ denote the union of hyperbolic regular leaves and put
$H_n=\{\varphi>-n\}\cap H$.  These sets are Borel, open along each leaf,
and increase to $H$.  Locality of one-variable Riesz measures gives
\[
 \mu_n|_{H_n}
   =T\wedge(A\omega-\Theta_0+\omega_{\mathrm P})|_{H_n}
   \ge (T\wedge\omega_{\mathrm P})|_{H_n}.
\]
Consequently
\[
 (T\wedge\omega_{\mathrm P})(H_n)
       \le A\langle T,\omega\rangle.
\]
Monotone convergence for the fixed positive measure
$T\wedge\omega_{\mathrm P}$ proves the asserted bound.  No limit of the
varying measures $\mu_n$ is needed.

If $T$ is nonzero, extremal, and has full trace mass on $H$, its
Poincar\'e mass is finite and strictly positive.  Divide $T$ by this
mass and apply \cite[Theorems~7.1 and~7.7]{DNS2012}.
Theorem~7.7 gives precisely the ambient-trace almost-everywhere conclusion
and the normalization $T/\langle T,\omega\rangle$.  Its weighted disk
current has mass
\[
 \int_{\mathbb D}\log^+\frac r{|\zeta|}\,\phi_a^*\omega
       =\int_0^r\frac{dt}{t}\int_{\mathbb D(t)}\phi_a^*\omega
       =\mathcal T_a(r),
\]
so it is exactly the Nevanlinna current in the statement.
\end{proof}

\section{A diffuse projective example with a uniform rate}
\label{sec:rate}

Fix the convention $\ddc=(i/\pi)\partial\bar\partial$.  Let $S$ be a compact
Riemann surface of genus at least two and let
\[
 A=\mathbb C/(\mathbb Z+i\mathbb Z).
\]
Choose a Poincar\'e form $\omega_S$ whose lift to $\D$ is
\begin{equation}
 \omega_\D=\ddc[-\log(1-\abs{z}^2)]
 =\frac{i}{\pi}\frac{dz\wedge d\bar z}{(1-\abs{z}^2)^2}.                                  \label{eq:omegaD}
\end{equation}
Pick $0\neq\eta_0\in H^0(S,K_S)$.  Some closed loop $\gamma_0$ has
$p_0=\int_{\gamma_0}\eta_0\neq0$, since otherwise $\eta_0$ would have a
single-valued holomorphic primitive on compact $S$.  Put
\[
 \xi=\sqrt2+i\sqrt3,
 \qquad \eta=\frac \xi {p_0}\eta_0.
\]
On the projective surface $M=S\times A$, let $dw$ be the global holomorphic
one-form from $A$ and define
\begin{equation}
 \alpha=dw-\eta,
 \quad \F=\ker\alpha,
 \quad \beta=i\alpha\wedge\bar\alpha,
 \quad \Omega=\omega_S+\beta,
 \quad V=\int_M\beta\wedge\omega_S,
 \quad T=\frac\beta V.                                                                  \label{eq:construction}
\end{equation}

\begin{proof}[Proof of Theorem~\ref{main:rate}]
\emph{Geometry and mass.}
The one-form $\alpha$ is closed and nowhere zero because
$\alpha(\partial_w)=1$.  Thus its kernel is integrable and nonsingular.
The two positive summands in $\Omega$ control respectively the base and the
remaining vertical direction, so $\Omega$ is K\"ahler and
$\Omega|_{T\F}=\omega_S$.  Along a loop $\gamma$, leaf holonomy on $A$ is
translation by $\int_\gamma\eta$.  In particular $\gamma_0$ translates by
$\xi$.  Since $1,\sqrt2,\sqrt3$ are linearly independent over $\mathbb Q$, the
cyclic subgroup generated by $\xi$ is dense in $A$.  Every leaf is therefore
dense.  Its projection to $S$ is a covering, so its universal cover is $\D$.

The form $\beta$ is positive, closed, decomposable, and
$\beta\wedge\alpha=0$.  Hence $T$ is positive, closed, and directed.  Since
$\beta^2=0$,
\[
 \pair{T}{\Omega}=V^{-1}\int_M\beta\wedge\omega_S=1.
\]
Every term of $\beta$ containing $\eta$ wedges to zero with $\omega_S$ on the
one-dimensional base, so
\begin{equation}
 \beta\wedge\omega_S=i\,dw\wedge d\bar w\wedge\omega_S.                                 \label{eq:product-volume}
\end{equation}
Thus the trace measure is normalized smooth product volume and $T$ is
diffuse.

\emph{Extremality.}
Suppose $T=T_1+T_2$ with $T_j$ positive, directed, and $\ddc$-closed.  Rank-one
positivity and $0\leq T_j\leq T$ imply $T_j=h_jT$ for measurable
$0\leq h_j\leq1$.  Locally $\alpha=dt$ in holomorphic foliated coordinates
$(z,t)$, and
\[
 \ddc\bigl(h_j\,i\,dt\wedge d\bar t\bigr)=0
\]
is precisely weak leafwise harmonicity of $h_j$.

Write $w=x+iy$, $\eta=a+ib$ with real closed harmonic one-forms $a,b$, and
Fourier-expand in $L^2(S\times A)$:
\[
 h=\sum_{k=(m,n)\in\mathbb Z^2}h_k(s)e^{2\pi i(mx+ny)}.
\]
On the $k$th mode, the leafwise differential is the unitary flat connection
\[
 D_k=d+2\pi i(ma+nb),\qquad P_k=D_k^*D_k.
\]
Here the adjoint is taken for the conformal Riemannian metric on $S$ whose area
form is $\omega_S$. Weak harmonicity gives $P_kh_k=0$, and elliptic regularity makes $h_k$ smooth.
For $k\neq0$, the holonomy around $\gamma_0$ is
$\exp(2\pi i(m\sqrt2+n\sqrt3))\neq1$; hence the flat line bundle has no
parallel section.  The energy identity
$\pair{P_kh_k}{h_k}=\norm{D_kh_k}_2^2$ gives $h_k=0$.  The zero mode is a
harmonic function on compact $S$ and is constant.  Fourier uniqueness makes
$h$ constant almost everywhere, proving extremality.

\emph{A global leafwise Poisson equation.}
Write the restriction of $\theta$ to $T\F$ as
\[
 \theta|_{T\F}=f_\theta\omega_S.
\]
By \eqref{eq:product-volume},
\[
 c:=\pair{T}{\theta}
 =\frac1V\int_M f_\theta\,\beta\wedge\omega_S
\]
is the product mean of $f_\theta$.  We claim that there is a smooth function
$u_\theta$ on $M$ satisfying
\begin{equation}
 \ddc_{\F}u_\theta=(f_\theta-c)\omega_S,
 \qquad
 \norm{u_\theta}_\infty\leq C\norm{\theta}_{C^{11}}.                                    \label{eq:poisson}
\end{equation}
For the zero fiber-Fourier mode the compatibility condition is exactly the
zero integral on $S$, so the ordinary Poisson equation is solvable.  The relation between the nonnegative Laplacian and our convention is
$dd^c_{\F}u=-(P_ku)\omega_S/(2\pi)$ on the $k$th mode. For a
nonzero mode $k=(m,n)$, set
$u_k=-2\pi P_k^{-1}(f_\theta)_k$; for the zero mode choose the
mean-zero solution with right-hand side $-2\pi((f_\theta)_0-c)$.
We record the quantitative point.
Let
\[
 \delta_k=\dist(m\sqrt2+n\sqrt3,\mathbb Z).
\]
If $\ell$ is a nearest integer, the nonzero algebraic integer
$m\sqrt2+n\sqrt3-\ell\in\mathbb Q(\sqrt2,\sqrt3)$ has nonzero integral norm.
Its other three conjugates are $O(1+\abs k)$, whence
\begin{equation}
 \delta_k\geq c(1+\abs k)^{-3}.                                                          \label{eq:diophantine}
\end{equation}
On the compact torus of unitary characters of $S$, the bottom eigenvalue of a
flat connection Laplacian satisfies
\[
 \lambda_0(\chi)\geq c\abs{\chi(\gamma_0)-1}^2.
\]
For completeness, choose harmonic connection one-forms $\upsilon$ near the
trivial character. The operator $(d+i\upsilon)^*(d+i\upsilon)$ has a
simple zero eigenvalue at $\upsilon=0$. Its first variation on constants
vanishes because $d^*\upsilon=0$; its quadratic eigenvalue term is
$\|\upsilon\|_2^2/\operatorname{Vol}(S)$. Analytic perturbation therefore
gives $\lambda_0\ge c\|\upsilon\|^2$ near zero. The character value on
$\gamma_0$ differs from $1$ by at most $C\|\upsilon\|$.
Away from the trivial character the bottom eigenvalue is positive,
and compactness of the character torus completes the estimate.  Together with
\eqref{eq:diophantine},
\begin{equation}
 \lambda_0(P_k)\geq c(1+\abs k)^{-6}.                                                     \label{eq:twisted-gap}
\end{equation}
The inverse therefore loses only a fixed polynomial in $k$.  More explicitly,
gauge each unitary character into a compact harmonic fundamental domain.  In
that gauge the coefficients of the connection Laplacians and the $H^2$
elliptic constants are uniform, while
\[
 \norm{u_k}_2\leq 2\pi\lambda_0(P_k)^{-1}\norm{(f_\theta)_k}_2
 \leq C(1+\abs k)^6\norm{(f_\theta)_k}_2.
\]
The uniform $H^2$ estimate and $H^2(S)\hookrightarrow C^0(S)$ give the same
sixth-order loss in that gauge.  The gauge is unitary and hence preserves pointwise absolute values.
In particular, the following weaker bound is sufficient:
\begin{equation}
 \norm{u_k}_\infty\leq
 C(1+\abs k)^8\norm{(f_\theta)_k}_2.                                                     \label{eq:inverse-bound}
\end{equation}
Eleven fiber derivatives give
$\norm{(f_\theta)_k}_2\leq
C(1+\abs k)^{-11}\norm{\theta}_{C^{11}}$; the resulting
$(1+\abs k)^{-3}$ series is summable on $\mathbb Z^2$.  This proves the
$C^0$ estimate in \eqref{eq:poisson}; the corresponding higher elliptic estimates have only polynomial losses in
$k$, since the product-gauge connection coefficients grow linearly in $k$.
The Fourier coefficients of a smooth $f_\theta$ decay faster than every
power, so the series and all its derivatives converge and the solution is smooth.

\emph{Jensen's formula.}
A universal cover has the form
\[
 \phi_a(z)=\bigl(\pi(z),w_0+H(z)-H(0)\bigr),
 \qquad dH=\pi^*\eta.
\]
Here $\pi\colon\D\to S$ is a universal covering centered at the
base point of $a$, and $w_0$ is its elliptic coordinate.
Hence $\phi_a^*\alpha=0$ and $\phi_a^*\Omega=\omega_\D$.  By the chosen
normalization and Jensen's formula,
\begin{equation}
 \mathcal T_a(r)=\int_{\abs z<r}\log\frac r{\abs z}\,\omega_\D=-\log(1-r^2).                       \label{eq:exact-denom}
\end{equation}
Moreover
$\phi_a^*(\theta-c\Omega)=\ddc(u_\theta\circ\phi_a)$, so
\[
 \pair{\tau_r^a}{\theta}-c
 =\frac{\displaystyle
       \frac1{2\pi}\int_0^{2\pi}u_\theta(\phi_a(re^{it}))\,dt-u_\theta(a)}
      {-\log(1-r^2)}.
\]
The numerator has absolute value at most $2\norm{u_\theta}_\infty$.
Consequently
\[
 \big|\langle\tau_r^a-T,\theta\rangle\big|
 \le\frac{2C\|\theta\|_{C^{11}}}{-\log(1-r^2)}.
\]
Absorbing $2$ into the geometric constant proves the required estimate.
\end{proof}

\section{Logarithmic growth on every hyperbolic leaf}
\label{sec:every-leaf}

The intrinsic metric lemma also proves the every-leaf assertion without
using the measure-theoretic part of Theorem~\ref{main:mass}.

\begin{proof}[Proof of Theorem~\ref{main:growth}]
Fix a center $a$ on a hyperbolic leaf and its universal covering.
Put $B=\sup_{M\setminus E}\varphi<\infty$ and choose $A>0$ with
$\Theta_0\le A\omega$.  Along the covering disk,
\eqref{eq:fm-curvature} gives
\[
 \omega_{\mathrm P}
       =dd^c(\varphi\circ\phi_a)+\phi_a^*\Theta_0.
\]
The function $\varphi\circ\phi_a$ is smooth and finite on each compact
subdisk.  Jensen's formula and \eqref{eq:fm-poincare-normalization} yield
\begin{align*}
 -\log(1-r^2)
 &=\frac1{2\pi}\int_0^{2\pi}
          \varphi(\phi_a(re^{it}))\,dt-\varphi(a)
       +\int_{\mathbb D}\log^+\frac r{|\zeta|}\,\phi_a^*\Theta_0\\
 &\le B-\varphi(a)+A\mathcal T_a(r).
\end{align*}
Take $c_0=A^{-1}$ and $C_a=A^{-1}(B-\varphi(a))\ge0$.
Since $-\log(1-r^2)\ge\log(1/(1-r))-\log2$, this also gives
\[
 \mathcal T_a(r)\ge\frac{c_0}{2}\log\frac1{1-r}
\]
for $r$ sufficiently close to $1$, with the threshold allowed to depend
on $a$. The proof has used only the leafwise metric and no directed
current, so it applies when no diffuse directed current exists.
\end{proof}

\section{Open problems}
\label{sec:further-cases}

The following problems ask for extensions of Theorems~\ref{main:mass},
\ref{main:rate}, and~\ref{main:growth}, respectively. They specify cases
not settled in this paper; we do not claim that every special case is
unknown in the literature.

\begin{openproblem}
\label{open:mass}
Let $\F$ be a saturated holomorphic curve foliation on a compact
K\"ahler manifold $(M,\omega)$, whose singular set $E$ may have positive
dimension. Is there a constant $C=C(M,\omega,\F)$ such that
\[
 \int_{M\setminus E}T\wedge\omega_{\mathrm P}
                       \le C\|T\|_\omega
\]
for every positive directed current $T$ with $dd^cT=0$ on $M$?
Under the same geometric hypotheses, does
\[
 \tau_r^a\longrightarrow T/\|T\|_\omega
 \quad\text{for }(T\wedge\omega)\text{-almost every }a
\]
hold whenever $T\ne0$ is extremal and its trace gives full mass to
hyperbolic regular leaves?
\end{openproblem}

Theorem~\ref{main:mass} answers both questions when $E$ is finite.
The proof here does not supply a reduction from a positive-dimensional
singular set to that case.

\begin{openproblem}
\label{open:rate}
Fix a saturated holomorphic curve foliation with finite singular set
on a compact K\"ahler manifold. Suppose that it admits a nonzero
extremal positive directed $dd^c$-closed current whose trace gives full
mass to hyperbolic regular leaves. Which geometric conditions ensure
the existence of such a current $T$, normalized to mass one, and
numbers $\sigma>0$ and $N\in\mathbb N$ such that, for
$(T\wedge\omega)$-almost every $a$, there are $C_a<\infty$ and $r_a<1$
with
\[
 \big|\langle\tau_r^a-T,\theta\rangle\big|
 \le C_a\|\theta\|_{C^N}[-\log(1-r^2)]^{-\sigma}
\]
for every smooth $(1,1)$-form $\theta$ and $r_a<r<1$?
The exponents should be controlled by the stated geometric conditions.
\end{openproblem}

This asks for a concrete rate on a foliation fixed in advance.
Theorem~\ref{main:rate} supplies one class of examples with $\sigma=1$
and $N=11$, and with a constant uniform in the center. Its proof uses
translation holonomy and a polynomial spectral lower bound, neither
of which is available for a general foliation. Related questions are
to obtain such estimates on projective space, for nonclosed extremal
currents, or in examples with positive transverse entropy. We do not
claim that $N=11$ is optimal.

\begin{openproblem}\label{open:growth}
Let $\F$ be a saturated holomorphic curve foliation on a compact
K\"ahler manifold, with a possibly positive-dimensional singular set.
Does every hyperbolic regular leaf satisfy
\[
 \mathcal T_a(r)\ge c_a\log\frac1{1-r}
\]
for some $c_a>0$ and all $r$ sufficiently close to $1$, without any
assumption on the existence of a diffuse directed current? Can one
obtain the stronger estimate
\[
 \mathcal T_a(r)\ge c_0[-\log(1-r^2)]-C_a,
 \qquad 0<r<1,
\]
with one positive constant $c_0$ independent of the leaf and center?
\end{openproblem}

The first question retains the growth conclusion in Problem~2.4;
the second asks for the uniform leading coefficient of
Theorem~\ref{main:growth}. Even in that theorem, the additive constant
$C_a$ and the threshold for \eqref{eq:original-growth} may depend on
the center. No uniform bound for them near the singular set is proved
here.

\bigskip
\noindent\textbf{Statement on the use of artificial intelligence.}
The research in August 2026 used
GPT-5.6-sol for approximately 10 hours. The writing is assisted by GPT-6-astra. The author writes the title, abstract, statement of the main Theorems, statement of open questions, and this statement on the use of AI.
The author has been searching for credits and references but only within limited time. The author is willing to update such important information upon request.

\end{document}